\documentclass{article}

\usepackage[latin1]{inputenc}
\usepackage{amssymb}
\usepackage{amsfonts}
\usepackage{amscd}
\usepackage{mathrsfs}
\usepackage[dvips]{graphicx}
\usepackage[T1]{fontenc}
\usepackage[english]{babel}
\usepackage{marvosym}
\usepackage{amsthm}
\usepackage{multicol}
\usepackage[all]{xy}

\usepackage{todonotes}

\usepackage{graphicx,tikz}

\usepackage{pb-diagram}

\usepackage{bm}

\usepackage{wasysym}

\usepackage{color}

\usetikzlibrary{decorations.pathreplacing,backgrounds,decorations.markings,arrows}
\tikzset{wei/.style={draw=red,double=red!40!white,double distance=1.5pt,thin}}
\tikzset{bdot/.style={fill,circle,color=blue,inner sep=3pt,outer sep=0}}
\tikzset{dir/.style={postaction={decorate,decoration={markings,mark=at position .8 with {\arrow[scale=1.3]{>}}}}}}

\theoremstyle{plain}
\newtheorem{thm}{Theorem}[section]
\newtheorem{lem}[thm]{Lemma}

\newtheorem{prop}[thm]{Proposition}
\newtheorem{conj}[thm]{Conjecture}
\newtheorem{coro}[thm]{Corollary}

\theoremstyle{definition}
\theoremstyle{remark}
\newtheorem{rk}[thm]{Remark}
\newtheorem{df}[thm]{Definition}

\def\bbC{\mathbb{C}}

\def\bbF{\mathbb{F}}

\def\bbP{\mathbb{P}}
\def\bbQ{\mathbb{Q}}

\def\bbZ{\mathbb{Z}}

\def\calE{\mathcal{E}}
\def\calF{\mathcal{F}}
\def\calG{\mathcal{G}}

\def\calO{\mathcal{O}}

\def\calQ{\mathcal{Q}}

\def\calX{\mathcal{X}} 
\def\calY{\mathcal{Y}}

\def\frakg{\mathfrak{g}}

\def\bfk{\mathbf{k}}

\def\bfv{\mathbf{v}}
\def\bfvv{\overline{\mathbf{v}}}

\def\bfw{\mathbf{w}}
\def\bfww{\overline{\mathbf{w}}}

\def\bfU{\mathbf{U}}

\def\bfV{\mathbf{V}}
\def\bfW{\mathbf{W}}

\newcommand\restr[2]{{% we make the whole thing an ordinary symbol
  \left.\kern-\nulldelimiterspace % automatically resize the bar with \right
  #1 % the function
  \vphantom{\big|} % pretend it's a little taller at normal size
  \right|_{#2} % this is the delimiter
  }}

\def\dimI{\dim_I}

\def\homo{\operatorname{\it \mathscr{H}\kern-.25em om}}
\def\ext{\operatorname{\it \mathscr{E}\kern-.25em xt}}
\def\edo{\operatorname{\it \mathscr{E}\kern-.25em nd}}
\def\der{\operatorname{\it \mathscr{D}\kern-.25em er}}

\def\mod{\mathrm{mod}}

\def\Hom{\mathrm{Hom}}

\def\Ext{\mathrm{Ext}}

\def\Id{\mathrm{Id}}

\def\Rep{\mathrm{Rep}}

\def\IC{\mathrm{IC}}

\def\eG{{\widehat\Gamma_d}}

\def\Repd{\Rep^0(\eG)}

\title{\huge Flag versions of quiver Grassmannians for Dynkin quivers have no odd cohomology}
\author{Ruslan Maksimau \\\\
AGM - Analyse, géométrie et modélisation,\\
CY Cergy Paris Université,\\
2 Av. Adolphe Chauvin, 95300 Pontoise, France\\
            ruslmax@gmail.com, ruslan.maksimau@cyu.fr
}

\date{}

\begin{document}
\maketitle
\setcounter{tocdepth}{2}

\begin{abstract}
We prove the conjecture that flag versions of quiver Grassmannians (also known as Lusztig's fibers) for Dynkin quivers (types $A$, $D$, $E$) have no odd cohomology groups over an arbitrary ring. Moreover, for types $A$ and $D$ we prove that these varieties have affine pavings. We also show that to prove the same statement for type $E$, it is enough to check this for indecomposable representations. 

We also give a flag version of the result of Cerulli Irelli-Esposito-Franzen-Reineke on rigid representations: we prove that flag versions of quiver Grassmannians for rigid representations have a diagonal decomposition. In particular, they have no odd cohomology groups.
\end{abstract}

\tableofcontents

\section{Introduction}

\subsection{Motivation}
The motivation of this paper comes from the study of KLR (Khovanov-Lauda-Rouquier) algebras in \cite{Mak-parity}. KLR algebras, also called quiver-Hecke algebras, were introduced by Khovanov-Lauda \cite{KL} and Rouquier \cite{Rouq-2KM} for categorification of quantum groups. Let $\Gamma$ be a quiver without loops and let $\nu$ be a dimension vector for this quiver. Let $\bfk$ be a field. Then we can construct  the KLR algebra $R_\nu$ over $\bfk$ associated to $\Gamma$ and $\nu$.

To motivate our resaerch, we recall some facts about the geometric study of KLR algebras. The reader can see \cite{Mak-parity} for more details.

A geometric construction of KLR algebras is given in \cite{Rouq-2KM} and \cite{VV} (over a field of characteristic zero). The positive characteristic case is done in \cite{Mak-parity}. Let us describe the idea of this geometric construction. There is a complex algebraic group $G$ and complex algebraic varieties $X$ and $Y$ (depending on $\Gamma$ and $\nu$) with a $G$-action and a proper $G$-invariant map $\pi\colon X\to Y$ such that the KLR algebra $R_\nu$ is isomorphic to the extension algebra $\Ext^*_G(\pi_*\underline\bfk_X, \pi_*\underline\bfk_X)$. Here $\underline\bfk_X$ is the constant $G$-equivariant sheaf on $X$. The sheaf $\pi_*\underline\bfk_X$ is the pushforward of $\underline\bfk_X$, this pushforward is an element of the $G$-equivariant bounded derived category of sheaves of $\bfk$-vector spaces on $Y$. Abusing the terminology, we say sometimes "a sheaf" for "a complex of sheaves".

For understanding the algebra $R_\nu$, it is important to understand the decomposition of the sheaf $\pi_*\underline\bfk_X$ into indecomposables.   Now, assume that $\Gamma$ is a Dynkin quiver. In this case the number of $G$-orbits in $Y$ is finite. Assume first that the characteristic of $\bfk$ is zero. We can deduce from the decomposition theorem \cite[Thm.~6.2.5]{BBD} and from \cite[Thm.~2.2]{Rein} that the indecomposable direct summands of $\pi_*\underline\bfk_X$ up to shift are exactly the sheaves of the form $\IC(\calO)$, where $\calO$ is a $G$-orbit in $Y$ and $\IC(\calO)$ is the simple perverse sheaf associated to the orbit $\calO$.
Now, we want to understand what happens if the characteristic of $\bfk$ is positive. In this case the decomposition theorem fails, so the theory of perverse sheaves does not help us to understand $\pi_*\underline\bfk_X$. However, Juteau-Mautner-Williamson \cite{JMW} introduce a new tool for this: parity sheaves. For the variety $Y$ as above, their construction yields some new sheaves $\calE(\calO)$ on $Y$. It happens, that in characteristic zero we have $\IC(\calO)=\calE(\calO)$. However, in positive characteristics, the sheaves $\calE(\calO)$ behave better than $\IC(\calO)$. In particular, \cite{JMW} gives a version of the decomposition theorem for parity sheaves that also works for positive characteristics. However, this version needs an extra-condition: the fibers of $\pi$ must have no odd cohomology groups over $\bfk$. Note that the fibers of the  map $\pi$ are isomorphic to flag versions of quiver Grassmannians. 

A version of the following conjecture is considered in \cite{Mak-parity}.

\begin{conj}
\label{conj:fibres}
Assume that $\Gamma$ is a Dynkin quiver. Then for each field $\bfk$ and each $y\in Y$ we have $H^{\rm odd}(\pi^{-1}(y),\bfk)=0$.
 \end{conj}
It is proved in \cite{Mak-parity} that if Conjecture \ref{conj:fibres} holds, then in any characteristic the indecomposable direct summands of $\pi_*\underline\bfk_X$, up to shift, are exactly the sheaves of the form $\calE(\calO)$.
Some other consequences of this conjecture are explained in \cite[Sec.~3.10]{Mak-parity}.
  Conjecture \ref{conj:fibres} implies that parity sheaves are stable under the convolution product and then we get a geometric construction (similar to Lusztig's construction in terms of perverse sheaves) of the quantum group $U^-_q(\frakg)$ for types $A$, $D$, $E$ in terms of parity sheaves. Without Conjecture \ref{conj:fibres}, it was only clear how to do the geometric construction of the coproduct but not of the product in $U^-_q(\frakg)$ using parity sheaves.
In particular, this yields a new and interesting basis in the quantum group $U^-_q(\frakg)$ for types $A$, $D$, $E$ in terms of parity sheaves.

Conjecture \ref{conj:fibres} is proved in \cite{Mak-parity} only in type $A$. It remained open for types $D$ and $E$.
An attempt to prove	Conjecture \ref{conj:fibres} is given in \cite{McNam}. The paper \cite{McNam} studies algebras of the form $\Ext^*_G(\pi_*\underline\bfk_X, \pi_*\underline\bfk_X)$ for more general $G$, $X$ and $Y$. This paper relates the condition ${H^{\rm odd}(\pi^{-1}(y),\bfk)=0}$ with the polynomial quasi-heredity of the algebra  
$\Ext^*_G(\pi_*\underline\bfk_X, \pi_*\underline\bfk_X)$. On the other hand, KLR algebras $R_\nu$ for Dynkin quivers are known to be polynomial quasi-hereditary due to \cite{BKM}.

In the present paper we prove Conjecture \ref{conj:fibres}. We don't use KLR algebras, we work directly with the geometric description of the fibers of $\pi$. These fibers are flag versions of quiver Grassmannians. Our result proves in fact more than Conjecture \ref{conj:fibres}: we also show that the fibers have no odd cohomology groups over an arbitrary ring $\bfk$, not obligatory a field. Moreover, in types $A$ and $D$ we also show that the fibers have affine pavings. Note that this result is also stronger than the similar result from \cite{Mak-parity} in type $A$.

There is an example of a situation where it is important to know that the fibers of $\pi$ have no odd cohomology over a ring (not only over a field). In the Williamson's counterexample \cite{Will} in type $A$ to the Kleshchev-Ram conjecture, he uses the modular system $(\bbQ_p, \bbZ_p, \bbF_p)$. In particular, for this, it is important to know that we have $H^{\rm odd}(\pi^{-1}(y),\bbZ_p)=0$. The results of the present paper allow us to use the same approach in types $D$ and $E$.

One more application of Conjecture \ref{conj:fibres} to Kato's reflection functors is given in \cite[Sec.~7]{McNam}.
Moreover, the main result of the present paper together with \cite[Sec.~4]{McNam} yields an alternative proof of the fact that KLR algebras of Dynkin ($A$, $D$, $E$) types are  polynomial quasi-hereditary.

\subsection{Results}
From now on, we assume that $\bfk$ is an arbitrary ring (not obligatory a field). We always assume that our quiver $\Gamma$ has no oriented cycles.

The fibers of $\pi$ have an explicit geometric description: they are flag versions of quiver Grassmannians for complete flag types. However, all the methods of the present paper work well for non-complete flags. So we will also allow non-complete flags below. 

Let $\Gamma=(I, H)$ be a quiver, where $I$ and $H$ are the sets of its vertices and arrows respectively. Fix an increasing sequence ${\bfv=(\bfv_1,\ldots,\bfv_d)}$ of dimension vectors and set $v=\bfv_d$. Let $\bfV$ be a representation of $\Gamma$ with dimension vector $v$. A flag version of the quiver Grassmannian is the following variety:
$$
\calF_\bfv(\bfV)=\{\bfV^1\subset \bfV^2\subset\ldots\subset \bfV^d=\bfV;~ \dimI\bfV^r=\bfv_r\},
$$
where each $\bfV^r$ is a subrepresentation of $\bfV$.
Quiver Grassmannians are special cases of $\calF_\bfv(\bfV)$ with $d=2$.
Note that the fibers $\pi^{-1}(y)$ discussed above are always of the form $\calF_\bfv(\bfV)$. The definition of our variety $\calF_\bfv(\bfV)$ coincides with what is usually called \emph{Lusztig's fibers} in the case when we consider full flags (or flag types where each step is concentrated at one vertex of the quiver). For other flag types, an extra assumption is sometimes made for the quotients $\bfV^r/\bfV^{r-1}$. We do not impose this assumption in our definition. It is explained in \cite{Zhou} how our approach can be adapted to the definition with an extra assumption.

\medskip
The main results of the paper are the following two theorems.
\begin{thm}
\label{thm:cell-intro}
Assume that $\Gamma$ is a Dynkin quiver of type $A$ or $D$. Then the variety $\calF_\bfv(\bfV)$ is either empty or has an affine paving.
\end{thm}

\begin{thm}
\label{thm:odd0-intro}
Assume that $\Gamma$ is a Dynkin quiver (types $A$, $D$, $E$). Then, for every ring $\bfk$, we have $H^{\rm odd}(\calF_\bfv(\bfV),\bfk)=0$.
\end{thm}

This paper is inspired by \cite{IEFR}, where these theorems are proved for quiver Grassmannians. Moreover, for quiver Grassmannians the analogue of Theorem \ref{thm:cell-intro} also holds for type $E$ and for affine quivers. The present paper gives generalizations of some results from \cite{IEFR} to flag versions. These generalizations are not straightforward, they need some new ideas that we explain below. 

Let us first summarize the idea of the proof in \cite{IEFR} for usual quiver Grassmannians. Let $\bfV$ be a representation of $\Gamma$ with dimension vector $v$ and let $v'$ be a dimension vector. The quiver Grassmannian is the following variety 
$$
\calG_{v'}(\bfV)=\{\bfV'\subset \bfV;~\dimI\bfV'=v', ~\mbox{$\bfV'$ is a subrepresentation of $\bfV$}\}.
$$
Let $0\to\bfV\to\bfU\to\bfW\to 0$ be a short exact sequence of representations of $\Gamma$. Set 
$$
\calG_{v',w'}=\{\bfU'\in \calG_{v'+w'}(\bfU);~\dimI(\bfU'\cap \bfV)=v'\}.
$$
The key step of their proof is the following proposition, see \cite[Thm.~26]{IEFR} and its proof.
\begin{prop}
\label{prop:key-recall}
If $\Ext^1_\Gamma(\bfW,\bfV)=0$, then
$$
\calG_{v',w'}(\bfU)\to\calG_{v'}(\bfV)\times \calG_{w'}(\bfW)
$$ 
is a vector bundle.
\end{prop}
For a Dynkin quiver, its indecomposable representations can be ordered in such a way that there are no extensions in one direction. So, Proposition \ref{prop:key-recall} reduces the statement to indecomposable representations. This allows us to prove Theorems \ref{thm:cell-intro} and \ref{thm:odd0-intro} for quiver Grassmannians in \cite{IEFR}.

\medskip
Now, let us explain how the present paper generalizes this approach for flag versions. We need to find a generalization of Proposition \ref{prop:key-recall} for flags. We start from the following observation. For each representation $\bfV$ of $\Gamma$ and each $\bfv$, the variety $\calF_\bfv(\bfV)$ can be seen as a quiver Grassmannian for some bigger quiver $\eG$ and some representation $\Phi(\bfV)$ of $\eG$. However, we cannot deduce our result directly from \cite{IEFR} because the quiver $\eG$ has no reason to be Dynkin. See also \cite[Thm.~3.27]{GLS} for a different approach realizing quiver flag varieties as quiver Grassmannians.

One more obstruction is that the natural functor $\Phi\colon \Rep(\Gamma)\to\Rep(\eG)$ does not preserve extensions. To fix this difficulty, we introduce a full subcategory $\Rep^0(\eG)$ of $\Rep(\eG)$ such that the image of $\Phi$ is in $\Rep^0(\eG)$ and  such that the functor $\Phi\colon \Rep(\Gamma)\to \Rep^0(\eG)$ preserves extensions. However, this creates a new obstruction: the category $\Rep^0(\Gamma)$ may have nonzero second extensions (while the category $\Rep(\eG)$ has no second extensions). The absence of second extensions in the category $\Rep(\Gamma)$ was an important point in the proof of Proposition \ref{prop:key-recall}. To overcome this problem, we prove some $\Ext$-vanishing properties of the category $\Rep^0(\eG)$ in Section \ref{subs-Ext-vanish}. This allows us to get an analogue of Proposition \ref{prop:key-recall}, see Proposition \ref{prop:Ext=0-bundle}.

This allows us to reduce the proof of Theorems \ref{thm:cell-intro} and \ref{thm:odd0-intro} to indecomposable representations. For types $A$ and $D$, the indecomposable representations are very easy to describe. So Theorem \ref{thm:cell-intro} can be done by hand for indecomposables. Type $E$ is more complicated. We don't know how to check Theorem \ref{thm:cell-intro} for indecomposables in this case. Note, however, that in type $E$ the number of cases to check is finite (we have a finite number of indecomposable representations and a finite number of flag types for each indecomposable representation).   

In type $E$, we manage to prove a weaker statement: Theorem \ref{thm:odd0-intro}. It is also enough to check it only for indecomposable representations. We do this in more generality. We prove the following result.
\begin{thm}
Let $\Gamma$ be an arbitrary quiver (without oriented cycles). Assume that $\bfV\in\Rep(\Gamma)$ satisfies $\Ext^1_\Gamma(\bfV,\bfV)=0$. Then $\calF_\bfv(\bfV)$ is either empty or has a diagonal decomposition. In particular, for every ring $\bfk$,  we have $H^{\rm odd}(\calF_\bfv(\bfV),\bfk)=0$.
\end{thm}
The theorem above generalizes a similar result from \cite{IEFR} about quiver Grassmannians. 

\bigskip
It is explained in \cite[Prop. A.1]{MakMin} how to extend Theorem \ref{thm:odd0-intro} to the Kronecker quiver $\bullet\rightrightarrows \bullet$.
	
After the preprint version of this paper was released, Xiaoxiang Zhou informed the author that the preprint \cite{Zhou} extends Theorem \ref{thm:cell-intro} to type $E$. Upon publication, this will confirm that flag versions of quiver Grassmannians
for Dynkin quivers ($A$, $D$, $E$) always have affine pavings (or they are empty). Additionally, \cite{Zhou} presents some results in the affine case. The preprint \cite{Zhou} uses methods from the present paper and from \cite{IEFR}.

\section{Flag versions of quiver Grassmannians}

We assume that all quivers have a finite number of vertices and arrows and have no oriented cycles. We also assume that all representations of quivers are finite dimensional and are over $\bbC$.

For a $\bbC$-algebra $A$ we denote by $\mod(A)$ the category of finite dimensional (over $\bbC$) left $A$-modules.

For integers $a$ and $b$ such that $a<b$ we set ${[a;b]=\{a,a+1,\ldots,b-1,b\}}$.

\subsection{Quivers}
Let $\Gamma=(I,H)$ be a quiver. We denote by $I$ and $H$ the sets of its vertices and arrows respectively. For each arrow $h\in H$ we write $s(h)$ and $t(h)$ for its source and target respectively.

\begin{df}
A \emph{dimension vector} $v$ for $\Gamma$ is a collection of positive integers $(v_i)_{i\in I}$. A \emph{representation} $\bfV$ of $\Gamma$ is a collection of finite dimensional complex vector spaces $\bfV_i$ for $i\in I$ and  a collection of linear maps $\bfV_{s(h)}\to \bfV_{t(h)}$ for each $h\in H$. We denote by $\Rep(\Gamma)$ the category of (finite dimensional) representations of $\Gamma$. We say that $v$ is the dimension vector of $\bfV$ if we have $\dim\bfV_i=v_i$ for each $i\in I$. In this case we write $\dimI \bfV=v$.
\end{df}
\subsection{Extended quiver}
Let $\Gamma=(I,H)$ be a quiver. Fix an integer $d>0$.

\begin{df}
The extended quiver $\eG=(\widehat I,\widehat H)$ is the quiver obtained from $\Gamma$ in the following way. The vertex set $\widehat I$ of $\eG$ is a union of $d$ copies of the vertex set $I$ of $\Gamma$, i.e., we have $\widehat I=I\times [1;d]$. The quiver $\eG$ has the following arrows:
\begin{itemize}
\item an arrow $(s(h),r)\to (t(h),r)$ for each $h\in H$ and each $r\in [1;d]$,
\item an arrow $(i,r)\to (i,r+1)$ for each $i\in I$ and $r\in [1;d-1]$.
\end{itemize}
\end{df}

\subsection{The functor $\Phi$}
Let $\Rep^0(\eG)$ be the full category of $\Rep(\eG)$ containing the objects $\bfV\in\Rep(\eG)$ that satisfy the following condition: for each $h\in H$ and each $r\in [1;d-1]$, the following diagram is commutative.
$$
\begin{CD}
\bfV_{(s(h),r)}   @>>> \bfV_{(t(h),r)}\\
@VVV                                             @VVV\\
\bfV_{(s(h),r+1)}   @>>>\bfV_{(t(h),r+1)}\\
\end{CD}
$$

Fix $r\in [1;d]$. Consider the following functor $\Phi\colon \Rep(\Gamma)\to \Rep^0(\eG)$. It is defined on objects in the following way. 
\begin{itemize}
\item For each $i\in I$ and $r\in [1;d]$, we have $\Phi(\bfV)_{(i,r)}=\bfV_i$. 
\item For each $h\in H$ and each $t\in [1;d]$, the map $\Phi(\bfV)_{(s(h),r)}\to \Phi(\bfV)_{(t(h),r)}$ is defined as $\bfV_{s(h)}\to \bfV_{t(h)}$.
\item  For each $i\in I$ and each $r\in [1;d-1]$, the map $\Phi(\bfV)_{(i,r)}\to \Phi(\bfV)_{(i,r+1)}$ is $\Id_{\bfV_i}$.
\end{itemize}
It does the following on morphisms. Each morphism $f\colon\bfV\to \bfW$ is sent to the morphism $\Phi(f)\colon \Phi(\bfV)\to \Phi(\bfW)$ such that for each $i\in I$ and $r\in [1;d]$ the component $\Phi(\bfV)_{(i,r)}\to \Phi(\bfW)_{(i,r)}$ of $\Phi(f)$ is just the component $\bfV_{i}\to\bfW_{i}$ of $f$.

\begin{lem}
\label{lem:Phi-ff}
The functor $\Phi$ is exact and fully faithful. 
\end{lem}
\begin{proof}
It is obvious from the construction that the functor is exact. It is also clear that the functor $\Phi$ is injective on morphisms. Let us prove that it is also surjective on morphisms.  

Consider a morphism $\phi\in \Hom_\eG(\Phi(\bfV),\Phi(\bfW))$. By definition, for each $i\in I$ and $r\in[1;d-1]$, we have the following commutative diagram.
$$
\begin{CD}
\Phi(\bfV)_{(i,r)}   @>{\phi_{(i,r)}}>> \Phi(\bfW)_{(i,r)}\\
@VVV                      @VVV\\
\Phi(\bfV)_{(i,r+1)}   @>{\phi_{(i,r+1)}}>> \Phi(\bfW)_{(i,r+1)}\\
\end{CD}
$$
After the identification $\bfV_i=\Phi(\bfV)_{(i,r)}=\Phi(\bfV)_{(i,r+1)}$ and $\bfW_i=\Phi(\bfW)_{(i,r)}=\Phi(\bfW)_{(i,r+1)}$, we get the diagram 
$$
\begin{CD}
\bfV_{i}   @>{\phi_{(i,r)}}>> \bfW_{i}\\
@V{\Id_{\bfV_i}}VV                      @V{\Id_{\bfW_i}}VV\\
\bfV_{i}   @>{\phi_{(i,r+1)}}>> \bfW_{i}\\
\end{CD}
$$
This shows that after the identification, the maps $\phi_{(i,r)}$ and $\phi_{(i,r+1)}$ are the same. This proves that $\phi$ is in the image of $\Phi$.
\end{proof}

\subsection{The functor $\Phi$ as a bimodule}
The category $\Rep(\Gamma)$ is equivalent to $\mod(\bbC\Gamma)$, where $\bbC\Gamma$ is the path algebra for the quiver $\Gamma$. Similarly, we have an equivalence of categories $\Rep(\eG)\simeq \mod(\bbC\eG)$. The subcategory $\Rep^0(\eG)$ of $\Rep(\eG)$ is equivalent to $\mod(Q)$, where $Q=\bbC\eG/J$ is the quotient of the path algebra $\bbC\eG$ by the ideal $J$ generated by the commutativity relations
$$
\begin{CD}
(s(h),r)   @>>> (t(h),r)\\
@VVV                                             @VVV\\
(s(h),r+1)   @>>>(t(h),r+1)\\
\end{CD}
$$ 
for each $h\in H$ and each $r\in [1,d-1]$. 

Let us give another description of $Q$. Let $L_d$ be the equioriented quiver of type $A_d$. More precisely, the quiver has $d$ vertices $1,2,\ldots,d$ and an arrow $r\to (r+1)$ for each $r\in[1;d-1]$. The set of vertices of the quiver $\eG$ can be considered as a direct product of sets of vertices of $\Gamma$ and $L_d$. It is easy to see that we have an identification $Q\simeq\bbC\Gamma\otimes \bbC L_d$. The isomorphism above allows us to consider $\bbC\Gamma$ and $\bbC L_d$ as subalgebras of $Q$.

Since $Q$ is a quotient of the path algebra $\bbC \eG$, it has idempotents $e_{(i,r)}$ for $i\in I$, $r\in [1;d]$ defined as the images of the standard idempotents of the path algebra. On the other hand, the algebra $\bbC\Gamma$ has idempotents $e_i$ for $i\in I$ and the algebra $\bbC L_d$ has idempotents $e_r$, $r\in [1;d]$. Then the isomorphism $Q\simeq\bbC\Gamma\otimes \bbC L_d$ identifies $e_{(i,r)}$ with $e_i\otimes e_r$.

Now, we describe the functor $\Phi\colon \mod(\bbC\Gamma)\to \mod(Q)$ in terms of bimodules. Abusing the notation, denote by $e_r$ the idempotent $1\otimes e_r\in Q$, in other words we have $e_r=\sum_{i\in I}e_{(i,r)}\in Q$. The following lemma is obvious. 
\begin{lem}
For each $\bfV\in\mod(\bbC\Gamma)$, we have $\Phi(\bfV)=Qe_1\otimes_{\bbC\Gamma}\bfV$.
\end{lem}

The description of the functor $\Phi$ above shows the following.
\begin{lem}
\label{lem:proj-to-proj}
The functor $\Phi\colon \Rep(\Gamma)\to\Rep^0(\eG)$ sends projective objects to projective objects and it also sends  injective objects to injective objects.
\end{lem}
\begin{proof}
	Let us abbreviate $A=\bbC\Gamma$ and $B=\bbC L_d$. We have $Q\simeq A\otimes B$. The $(Q,A)$-bimodule $Qe_1$ is in fact isomorphic to $A\otimes Be_1$. So, the functor $\Phi=Qe_1\otimes_A \bullet$ takes an $A$-module $M$ to the $Q$-module $M\otimes Be_1$.
	
	Let us check that $\Phi$ sends projectives to projectives. All indecomposable projective objects in $\mod(A)$ are direct summands of $A$. So, it is enough to check that $\Phi(A)$ is a projective $Q$-module. We have $\Phi(A)\simeq Qe_1$, it is obviously projective.
	
	Let us check that $\Phi$ sends injectives to injectives. All indecomposable injective objects in $\mod(A)$ are direct summands of $A^*$, where $A^*=\Hom_\bbC(A,\bbC)$ is the left $A$-module whose module structure is induced from the right $A$-module structure on $A$. So, it is enough to check that the $Q$-module $\Phi(A^*)$ is injective. We have $\Phi(A^*)=A^*\otimes Be_1$. This $Q$-module is injective because the $B$-module $Be_1$ is injective (because we have an isomorphism of left $B$-modules $B\simeq B^*$).

\end{proof}

Let us write $\Ext^i_\Gamma$ for the $i$th extension functor in the category $\Rep(\Gamma)$. We will also often need the extension functor in the category $\Repd$ (not in $\Rep(\eG)$). We will denote this extension functor $\Ext^i_Q$.

\begin{coro}
\label{coro:proj-inj-dim}
For each $\bfV, \bfW\in \Rep(\Gamma)$ and each $i\geqslant 0$, we have an isomorphism
$$
\Ext^i_Q(\Phi(\bfV),\Phi(\bfW))\simeq\Ext^i_\Gamma(\bfV,\bfW).
$$
Moreover, the object $\Phi(\bfV)\in \Repd$ has both projective and injective dimension at most one.
\end{coro}
\begin{proof}
This follows from Lemmas \ref{lem:Phi-ff} and \ref{lem:proj-to-proj}.
\end{proof}

\subsection{$\Ext$-vanishing properties}
\label{subs-Ext-vanish}

For each representation $\bfU\in\Repd=\mod(Q)$ and $r\in[1;d]$, set $\bfU_r=e_r\bfU$. We can consider $\bfU_r$ as a representation of $\Gamma$.

\begin{lem} 
\label{lem:Ext0,W',Phi(V)}
Assume $\bfV,\bfW\in \Rep(\Gamma)$ are such that $\Ext^1_\Gamma(\bfW,\bfV)=0$.
For each subrepresentation $\bfW'\subset \Phi(\bfW)$, we have $\Ext^1_Q(\bfW',\Phi(\bfV))=0$.
\end{lem}
\begin{proof}
Consider the short exact sequence
$$
0\to\bfW'\to\Phi(\bfW)\to\Phi(\bfW)/\bfW'\to 0
$$
and apply the functor $\Hom_Q(\bullet,\Phi(\bfV))$. We get an exact sequence 
$$
\Ext^1_Q(\Phi(\bfW),\Phi(\bfV))\to \Ext^1_Q(\bfW',\Phi(\bfV))\to \Ext^2_Q(\Phi(\bfW)/\bfW',\Phi(\bfV)).
$$
The left and right terms of this exact sequence are zero by Corollary \ref{coro:proj-inj-dim}. Then the middle term is also zero.

\end{proof}

It is well-known, that each representation $\bfW\in\Rep(\Gamma)$ has a two-step projective resolution. In particular, this implies, that we have $\Ext^i_\Gamma(\bfW,\bfV)=0$ for $i\geqslant 2$. The same is true for $\Rep(\eG)$, but not necessary for $\Repd$. However, the category $\Repd$ has a weaker property. 

\begin{lem}
\label{lem:Ext3=0}
For each $i\geqslant 3$ and $\bfV,\bfW\in\Repd$, we have $\Ext^i_Q(\bfW,\bfV)=0$.
\end{lem}
\begin{proof}
As explained above, the algebras $\bbC\Gamma$ and $\bbC L_d$ have projective dimensions at most $1$. Then the algebra $Q=\bbC\Gamma\otimes\bbC L_d$ has projective dimension at most $2$.
\end{proof}

\begin{lem}
\label{lem:Ext2=0,W'}
For each $\bfW\in \Rep(\Gamma)$, each subrepresentation $\bfW'\subset \Phi(\bfW)$,  each $\bfV'\in\Repd$ and $i\geqslant 2$, we have $\Ext^i_Q(\bfW',\bfV')=0$. 
\end{lem}

\begin{proof}
Apply the functor $\Hom_Q(\bullet,\bfV')$ to the short exact sequence 
$$
0\to \bfW'\to\Phi(\bfW)\to \Phi(\bfW)/\bfW'\to 0.
$$ 
We get an exact sequence 
$$
\Ext^i_Q(\Phi(\bfW),\bfV')\to\Ext^i_Q(\bfW',\bfV')\to \Ext^{i+1}_Q(\Phi(\bfW)/\bfW',\bfV').
$$
The left term is zero by Corollary \ref{coro:proj-inj-dim} and the right term is zero by Lemma \ref{lem:Ext3=0}. This implies that the middle term is also zero.
\end{proof}

\begin{coro}
\label{cor:Ext2=0, W',Phi(V)/V'}
Let $\bfV,\bfW\in \Rep(\Gamma)$ be such that $\Ext^1_\Gamma(\bfW,\bfV)=0$. Then for each representation $\bfV'\subset \Phi(\bfV)$ and each representation $\bfW'\subset \Phi(\bfW)$ we have $\Ext^1_Q(\bfW',\Phi(\bfV)/\bfV')=0$.
\end{coro}
\begin{proof}
Apply the functor $\Hom_Q(\bfW',\bullet)$ to the short exact sequence 
$$
0\to \bfV'\to\Phi(\bfV)\to \Phi(\bfV)/\bfV'\to 0.
$$ 
We get an exact sequence 
$$
\Ext^1_Q(\bfW',\Phi(\bfV))\to \Ext^1_Q(\bfW',\Phi(\bfV)/\bfV')\to \Ext^2_Q(\bfW',\bfV').
$$
The left term is zero by Lemma \ref{lem:Ext0,W',Phi(V)} and the right term is zero by Lemma \ref{lem:Ext2=0,W'}. Then the middle term is also zero.
\end{proof}

\subsection{Some exact sequences}
Fix $\bfW\in \Repd$. 

\begin{lem}
There is a short exact sequence 

$$
0\to\bigoplus_{r\in[1;d-1]}Qe_{r+1}\otimes_{\bbC\Gamma} \bfW_r{\to}\bigoplus_{i\in I,r\in [1;d]}Qe_r\otimes_{\bbC\Gamma} \bfW_r{\to}\bfW\to 0.
$$ 
\end{lem}
\begin{proof}
The proof is very similar to \cite[p.7]{CB}.
\end{proof}

\begin{coro}
\label{coro:long_ex_seq-Q}
For each $\bfV,\bfW\in \Repd$, we have a long exact sequence 
$$
0\to\Hom_\eG(\bfW,\bfV)\to \bigoplus_{r\in[1;d]}\Hom_{\Gamma}(\bfW_r,\bfV_r)
$$
$$
\to\bigoplus_{r\in[1;d-1]} \Hom_{\Gamma}(\bfW_r,\bfV_{r+1})\to \Ext^1_Q(\bfW,\bfV).
$$
\end{coro}
\begin{proof}
We apply the functor $\Hom_Q(\bullet,\bfV)$ to the short exact sequence above. 
\end{proof}

\subsection{Flags}

Let $\bfV=\bigoplus_{i\in I}\bfV_i$ be a representation of $\Gamma$ with dimension vector $v=(v_i)_{i\in I}$. 

\begin{df}
A \emph{flag type} of weight $v$ is a $d$-tuple of dimension vectors $\bfv=(\bfv_1,\ldots,\bfv_d)$, $\bfv_r=(\bfv_{r,i})_{i\in I}$ such that $\bfv_d=v$ and for each $i\in I$ we have $\bfv_{1,i}\leqslant \bfv_{2,i}\leqslant\ldots\leqslant \bfv_{d,i}$.

A \emph{flag} in $\bfV$ of type $\bfv$ is a sequence of subrepresentations $\bfV^1\subset \bfV^2\subset\ldots\subset \bfV^d=\bfV$ of $\bfV$ such that $\dimI\bfV^r=\bfv_r$.

Denote by $\calF_\bfv(\bfV)$ the set of all flags of type $\bfv$ in $\bfV$. This set has an obvious structure of a projective algebraic variety over $\bbC$. 
\end{df} 
The notion above generalizes quiver Grassmannians. 
\begin{df}
Let $\bfV$ be a representation of $\Gamma$ and lets $v'$ be some dimension vector. The quiver Grassmannian is the following variety 
$$
\calG_{v'}(\bfV)=\{\bfV'\subset \bfV;~\dimI\bfV'=v', ~\mbox{$\bfV'$ is a subrepresentation of $\bfV$}\}.
$$
\end{df}
Set $v=\dimI\bfV$. We clearly have, $\calG_{v'}(\bfV)\simeq \calF_{(v',v)}(\bfV)$. However, it is also always possible to realize the variety $\calF_{\bfv}(\bfV)$ as a quiver Grassmannian for the representation $\Phi(\bfV)$ of $\eG$. This works in the following way. The flag type $\bfv=(\bfv_1,\bfv_2,\ldots,\bfv_d)$ can be seen as a dimension vector $\alpha=(\alpha_{(i,r)})_{(i,r)\in \widehat I}$ for the quiver $\eG$ given by $\alpha_{(i,r)}=\bfv_{r,i}$. The following statement is obvious from the definitions.

\begin{prop}
	\label{prop:quiv-flag-via-quiv-Gr}
	There is an isomorphism of algebraic varieties $\calF_\bfv(\bfV)\simeq \calG_\alpha(\Phi(\bfV))$.
\end{prop}
\begin{proof}
	Taking a subrepresentation with dimension vector $\alpha$ of $\Phi(\bfV)$ means that for each $(i,r)\in \widehat I$ we take a vector subspace $\bfU_{(i,r)}$ of $\bfV_i$ with extra conditions on them. The conditions coming from the arrows of the form $(s(h),r)\to (t(h),r)$ imply that $ \bfU_r:=\bigoplus_{i\in I}\bfU_{(i,r)}$ is a subrepresentation of $\bfV$ with dimension vector $\bfv_r$ and the conditions coming from the arows of the form $(i,r)\to (i,r+1)$ imply $\bfU_r\subset \bfU_{r+1}$.
\end{proof}

\subsection{Reduction to the indecomposable case}
We set ${<w,v>=\sum_{i\in I}w_i\cdot v_i-\sum_{h\in H}w_{s(h)}\cdot v_{t(h)}}$, where $v$ and $w$ are dimension vectors.
We start from the following well-known lemma, see \cite[\S 1]{CB}.
\begin{lem}
Let $\bfV, \bfW\in\Rep(\Gamma)$ be two representations. Let $v$ and $w$ be their dimension vectors. 
Then we have an exact sequence 
{\fontsize{8.5pt}{9pt}\selectfont
$$
0\to\Hom_\Gamma(\bfW,\bfV)\to \bigoplus_{i\in I}\Hom(\bfW_i,\bfV_i)\to\bigoplus_{h\in H}\Hom(\bfW_{s(h)},\bfV_{t(h)})\to \Ext^1_\Gamma(\bfW,\bfV)\to 0.
$$
}
In particular, we have
$$
\dim\Hom_\Gamma(\bfW,\bfV)-\dim\Ext_\Gamma^1(\bfW,\bfV)=<w,v>.%=\sum_{i\in I}w_i\cdot v_i-\sum_{h\in H}w_{h'}\cdot v_{t(h)}
$$ 
\end{lem}

\begin{coro}
\label{cor:dimHom-Ext0}
If additionally we have $\Ext^1_\Gamma(\bfW,\bfV)=0$, then we have a short exact sequence
\begin{equation}
\label{eq:ses-Hom-quiver}
0\to\Hom_\Gamma(\bfW,\bfV)\to \bigoplus_{i\in I}\Hom(\bfW_i,\bfV_i)\to\bigoplus_{h\in H}\Hom(\bfW_{s(h)},\bfV_{t(h)})\to 0
\end{equation}
In particular, we have
$\dim\Hom_\Gamma(\bfW,\bfV)=<w,v>$.
\end{coro}

Let $\bfV$ and $\bfW$ be representations of $\Gamma$ with dimension vectors $v$ and $w$ respectively. Let $\bfv$ and $\bfw$ be flag types (with respect to $v$ and $w$ respectively). Assume that $0\to \bfV\to\bfU\stackrel{\pi}{\to}\bfW\to 0$ is a short exact sequence in $\Rep(\Gamma)$.
For each flag $\phi$ in $\bfU$ we denote by $\phi\cap \bfV$ the flag in $\bfV$ obtained by the intersection of components of $\phi$ with $\bfV$ and we denote by $\pi(\phi)$ the flag in $\bfW$ obtained by the images in $\bfW$ of the components of $\phi$. Set 
$$
\calF_{\bfv,\bfw}(\bfU)=\{\phi\in \calF_{\bfv+\bfw}(\bfU);~\phi\cap V\in\calF_{\bfv}(\bfV)\}.
$$
It is clear from the definition that if $\phi$ is in $\calF_{\bfv,\bfw}(\bfU)$, then $\pi(\phi)$ is in $\calF_{\bfw}(\bfW)$.

For a flag type $\bfv$, we set $\bfvv_r=\bfv_r-\bfv_{r-1}$ for $r\in[2;d]$ and $\bfvv_1=\bfv_1$. Similarly, we define $\bfww_r$ for a flag type $\bfw$.
\begin{prop}
\label{prop:Ext=0-bundle}
Assume that we have $\Ext^1_\Gamma(\bfW,\bfV)=0$. %and $\bfU=\bfV\oplus\bfW$. 
Then 
$$
\calF_{\bfv,\bfw}(\bfU)\to\calF_\bfv(\bfV)\times \calF_\bfw(\bfW), \qquad \phi\mapsto (\phi\cap \bfV,\pi(\phi))
$$
is a vector bundle of rank 
$$
\sum_{r=1}^{d-1}\sum_{t=r+1}^d <\bfvv_r,\bfww_t>.
$$
\end{prop}
\begin{proof}
Fix an isomorphism $\bfU\simeq \bfV\oplus\bfW$ in $\Rep(\Gamma)$. This is possible by the assumption on $\Ext^1$. 

First, we want to understand the fibers of the given map.
Fix $(\phi_\bfV,\phi_\bfW)\in \calF_\bfv(\bfV)\times \calF_\bfw(\bfW)$, where 
$\phi_\bfV=(\{0\}\subset \bfV^1\subset \bfV^2\subset\ldots\subset \bfV^d=\bfV)$ and $\phi_\bfW=(\{0\}\subset \bfW^1\subset \bfW^2\subset\ldots\subset \bfW^d=\bfW)$.
Let us describe the fiber of $(\phi_\bfV,\phi_\bfW)$.
For each $r\in[1;d]$, we want to construct a subrepresentation $\bfU^r$ of $\bfU$ such that $\bfU^r\cap \bfV=\bfV^r$ and $\pi(\bfU^r)=\bfW^r$. The choices of such a subrepresentation are parameterized by $\Hom_\Gamma(\bfW^r,\bfV/\bfV^r)$. Indeed, to each map $f\in \Hom_\Gamma(\bfW^r,\bfV/\bfV^r)$ we can associate a subrepresentation $\bfU^r\subset \bfV\oplus \bfW$ generated by the elements $(v,w)\in \bfV\oplus \bfW^r$ such that the image of $v$ in $\bfV/\bfV^r$ is $f(w)$.

 Moreover, for each $r\in[1;d-1]$, we must have $\bfU^r\subset \bfU^{r+1}$. This condition is equivalent to the commutativity of the following diagram
$$
\begin{CD}
\bfW^r @>>>\bfV/\bfV^r\\
@VVV       @VVV\\
\bfW^{r+1} @>>>\bfV/\bfV^{r+1}.
\end{CD}
$$ 

To sum up, a point of the fiber of $(\phi_\bfV,\phi_\bfW)$ is described by a family of homomorphisms $\Hom_\Gamma(\bfW^r,\bfV/\bfV^r)$ for $r\in [1;d]$ such that for each $r\in [1;d-1]$ the diagram above commutes.

Now, we change the point of view to describe this fiber in a different way. We can consider the flag $\phi_\bfV$ as a subrepresentation $\bfV'=\bigoplus_{r\in[1;d],i\in I}\bfV'_{(r,i)}$ of $\Phi(\bfV)$, see Proposition \ref{prop:quiv-flag-via-quiv-Gr}. The component $\bfV^r$ of the flag $\phi_\bfV$ is now considered as a part $\bfV'_r=\bigoplus_{i\in I}\bfV'_{(r,i)}$ of the representation $\bfV'$ of $\eG$. The flag type $\bfv$ can be considered as a dimension vector for the quiver $\eG$. Moreover, the representation $\bfV'$ of $\eG$ has dimension vector $\bfv$. Similarly, we consider the flag $\phi_\bfW$ as a subrepresentation $\bfW'$ of $\Phi(\bfW)$ with dimension vector $\bfw$. Then the fiber of $(\phi_\bfV,\psi_\bfW)$ is simply $\Hom_\eG(\bfW',\Phi(\bfV)/\bfV')$.

We have an obvious inclusion of vector spaces $\Hom_\eG(\bfW',\Phi(\bfV)/\bfV')\subset \bigoplus_{r\in[1;d]}\Hom_{\Gamma}(\bfW'_r,\bfV/\bfV'_r)$. Let us show that $\calF_{\bfv,\bfw}(\bfU)$ is a subbundle of the vector bundle on $\calF_\bfv(\bfV)\times \calF_\bfw(\bfW)$ with fiber $\bigoplus_{r\in[1;d]}\Hom_{\Gamma}(\bfW'_r,\bfV/\bfV'_r)$. For this, it is enough to present $\calF_{\bfv,\bfw}(\bfU)$ as a kernel of a morphism of vector bundles of constant rank. 

Indeed, since we assumed $\Ext^1_\Gamma(\bfW,\bfV)=0$, Corollary \ref{cor:Ext2=0, W',Phi(V)/V'} implies that we have $\Ext^1_Q(\bfW',\Phi(\bfV)/\bfV')=0$. Then the exact sequence in Corollary \ref{coro:long_ex_seq-Q} yields the short exact sequence
\begin{equation}
\label{eq:ses-bundle-final}
0\to\Hom_\eG(\bfW',\Phi(\bfV)/\bfV')\to \bigoplus_{r\in[1;d]}\Hom_{\Gamma}(\bfW'_r,\bfV/\bfV'_r)
\end{equation}
$$
\to\bigoplus_{r\in[1;d-1]} \Hom_{\Gamma}(\bfW'_r,\bfV/\bfV'_{r+1})\to 0.
$$

This short exact sequence implies that $\calF_{\bfv,\bfw}(\bfU)$ is a kernel of a surjective morphism of vector bundles. In particular, it is also a vector bundle. (Note that the vector bundles with fibers $\Hom_{\Gamma}(\bfW'_r,\bfV/\bfV'_r)$ and $\Hom_{\Gamma}(\bfW'_r,\bfV/\bfV'_{r+1})$ are defined via the same procedure. They are kernels of surjective morphisms of vector bundles coming from the short exact sequence (\ref{eq:ses-Hom-quiver}). See also the proof of \cite[Thm.~26]{IEFR}.)

Let us calculate the rank of this vector bundle. It is clear from the short exact sequence above that it is equal to 
$$
\sum_{r=1}^d \dim\Hom_{\Gamma}(\bfW'_r,\bfV/\bfV'_r)-\sum_{r=1}^{d-1} \dim\Hom_{\Gamma}(\bfW'_r,\bfV/\bfV'_{r+1}).
$$
By Corollary \ref{cor:dimHom-Ext0}, this rank is equal to 
$$
\sum_{r=1}^d<\bfw_r,v-\bfv_r>-\sum_{r=1}^{d-1}<\bfw_r,v-\bfv_{r+1}>~=~\sum_{r=1}^{d-1}<\bfw_r,\bfvv_{r+1}>
$$
$$
=\sum_{r=1}^{d-1}\sum_{t=r+1}^d<\bfww_r,\bfvv_{t}>.
$$
\end{proof}

\subsection{Affine paving}
\begin{df}
Let $X$ be a complex algebraic variety. We say that $X$ admits an affine paving if $X$ has a finite partition $X=X_1\coprod X_2\coprod \ldots \coprod X_n$ such that 
\begin{enumerate}
\item for each $1\leqslant r \leqslant n $, the union $X_1\coprod \ldots \coprod X_r$ is closed,
\item each $X_r$ is isomorphic to an affine space.
\end{enumerate}
\end{df}

\begin{prop}
\label{prop:reduction-indec}
Let $\Gamma$ be a Dynkin quiver. Let $\bfV$ be a representation of $\Gamma$. Let us decompose it in a direct sum of indecomposable representations
$$
\bfV=\bfV_1^{\oplus n_1}\oplus\ldots\oplus  \bfV_k^{\oplus n_k}.
$$  

Assume that for each $r\in[1;k]$ and each flag type $\bfv^r$ of weight $(\dimI\bfV_r)$, the variety $\calF_{\bfv^r}(\bfV_r)$ is either empty or has an affine paving. Then for each flag type of weight $(\dimI\bfV)$, the variety $\calF_\bfv(\bfV)$ is either empty or has an affine paving.

\end{prop}
\begin{proof}
We can assume that $\bfV_1,\ldots, \bfV_k$ are ordered in such a way that we have $\Ext^1(\bfV_r,\bfV_t)=0$ if $r\leqslant t$.
Then the statement follows easily from Proposition \ref{prop:Ext=0-bundle} by induction. 
\end{proof}

\subsection{Types $A$ and $D$}

The goal of this section is to prove the following theorem.

\begin{thm}
\label{thm:main}
Assume that $\Gamma$ is a Dynkin quiver of type $A$ or $D$. Then the variety $\calF_\bfv(\bfV)$ is either empty or has an affine paving.
\end{thm}
\begin{proof}
By Proposition \ref{prop:reduction-indec}, it is enough to prove the statement for indecomposable representations. This is done in two lemmas below.
\end{proof}

\begin{lem}
Assume that $\Gamma$ is a Dynkin quiver of type $A$ and that $\bfV$ is an indecomposable representation. Then the variety $\calF_\bfv(\bfV)$ is either empty or is a singleton.
\end{lem}
\begin{proof}
The statement follows from the fact that for each $i\in I$ the dimension of $\bfV_i$ is $0$ or $1$.
\end{proof}

\begin{lem}
Assume that $\Gamma$ is a Dynkin quiver of type $D$ and that $\bfV$ is an
indecomposable representation. Then the variety $\calF_\bfv(\bfV)$ is either empty or is a singleton or is a direct product of some copies of $\mathbb P^1_{\bbC}$.
\end{lem}
\begin{proof}
For each $i\in I$ the dimension of $\bfV_i$ is $0$, $1$ or $2$. Note that for each arrow $i\to j$ of the quiver such that $\dim \bfV_i=\dim \bfV_j=2$, the map $\bfV_i\to \bfV_j$ is bijective. In view of this, we can fix an identification $\bfV_i\simeq \bbC^{\dim \bfV_i}$ for each $i\in I$ such that the following condition is satisfied: for each arrow $i\to j$ of the quiver such that $\dim \bfV_i=\dim \bfV_j=2$, the corresponding map $\bbC^2\to \bbC^2$ is the identity map.

Then it is clear that the variety $\calF_\bfv(\bfV)$ is naturally included into a direct product of $\mathbb P^1_{\bbC}$. Indeed, a point of $\calF_\bfv(\bfV)$ is given by a choice of $1$-dimensional subspaces $\bfV'_i$ inside of some $2$-dimensional $\bfV_i$'s, these choices must satisfy some list of conditions. We may have the following types of conditions.
\begin{itemize}
\item For a given $i$ with $\dim\bfV_i=2$, we may have a condition that the $1$-dimensional subspace $\bfV'_i$ of $\bfV_i\simeq \bbC^2$ is equal to a fixed $1$-dimensional subspace of $\bbC^2$ (i.e., this condition imposes some choice of $\bfV'_i$ inside of some two-dimensional $\bfV_i$).  Such a condition may appear from an arrow of the form $\bbC^2\to \bbC$ or $\bbC\to\bbC^2$.
\item For some arrow $i\to j$ such that $\dim\bfV_i=\dim\bfV_j=2$, we may have a condition that $\bfV_{i}\to \bfV_{j}$ sends $\bfV'_{i}$ to $\bfV'_{j}$. When we identify $\bfV_i$ and $\bfV_j$ with $\bbC^2$ as explained above, this condition says that $\bfV'_{i}$ and $\bfV'_{j}$ are the same subspaces of $\bbC^2$.
\item We may have an impossible condition, implying that $\calF_\bfv(\bfV)$ is empty. 
\end{itemize}

To sum up, we get the following. We see that we can realize $\calF_\bfv(\bfV)$ as a subvariety of $(\bbP^1_\bbC)^m$ for some integer $m\geqslant 0$. Denote by $(x_1,x_2,\ldots,x_m)$ an element of $(\bbP^1_\bbC)^m$. The conditions that define $\calF_\bfv(\bfV)$ inside $(\bbP^1_\bbC)^m$ may have the following form.
\begin{itemize}
	\item We may have a condition that some $x_r$ must be equal to a given element of $\bbP^1_\bbC$.
	\item We may have a condition that some $x_r$ must be equal to some $x_t$.
	\item We may have an impossible condition.
\end{itemize}

It is clear, that $\calF_\bfv(\bfV)\subset(\bbP^1_\bbC)^m$ given by such conditions is either empty or is isomorphic to $(\bbP^1_\bbC)^{m'}$ for some $0\leqslant m'\leqslant m$. This proves the statement.
\end{proof}

\begin{rk}
\label{rk:typeE}
It is clear from the proof of Theorem \ref{thm:main} that for proving an analogue of this theorem for type $E$, it is enough to prove it for indecomposable representations. 

We are going to prove a weaker statement in type $E$: the variety $\calF_\bfv(\bfV)$ has no odd cohomology over an arbitrary ring. The same argument as above shows that for proving this statement, it is enough to check this only for indecomposable representations.
\end{rk}

\section{Rigid representations}

Let $\bfk$ be a ring.
The goal of this section is to show that the variety $\calF_\bfv(\bfV)$ has no odd cohomology over $\bfk$. It is enough to check this statement only for indecomposable representations, see Remark \ref{rk:typeE}. In fact, we are going to prove this fore more general class of representations. It is well-known that for each indecomposable representation $\bfV$ of a Dynkin quiver we have $\Ext^1_\Gamma(\bfV,\bfV)=0$, see for example the proof of \cite[Thm.~1]{CB}. We will prove that if some representation of a quiver satisfies this condition, then the variety $\calF_\bfv(\bfV)$ has no odd cohomology over $\bfV$. Moreover, we will prove that in this case the variety $\calF_\bfv(\bfV)$ has a diagonal decomposition if it is not empty.

\subsection{Rigid representations}
Let $\Gamma$ be an arbitrary quiver.

\begin{df}
We say that a representation $\bfV$ of $\Gamma$ is \emph{rigid} if ${\Ext^1_\Gamma(\bfV,\bfV)=0}$.
\end{df}

Choose $i\in I$. For each flag type $\bfv$, the $d$-tuple $\bfv_i=(\bfv_{i,1},\bfv_{i,2},\ldots,\bfv_{i,d})$ can be seen as a flag type for a quiver with one vertex. The vector space $\bfV_i$ can be seen as a representation of this quiver. The variety $\calF_{\bfv_i}(\bfV_i)$ is a usual (non-complete) flag variety. The variety $\calF_\bfv(\bfV)$ is obviously included to a direct product of usual (non-complete) flag varieties in the following way:
$$
\calF_\bfv(\bfV)\subset \prod_{i\in I}\calF_{\bfv_i}(\bfV_i).
$$ 

Set $\Rep_v(\Gamma)=\bigoplus_{h\in H}\Hom(\bbC^{v_{s(h)}},\bbC^{v_{t(h)}})$. We can see an element $X\in\Rep_v(\Gamma)$ as a representation of $\Gamma$ with dimension vector $v$.

Let $\calQ$ be the subvariety of $\Rep_v(\Gamma)\times \prod_{i\in I}\calF_{\bfv_i}(\bbC^{v_i})$ given by
$$
\calQ=\{(X,\phi);~X\in \Rep_v(\Gamma),~\phi\in \calF_\bfv(X)\}.
$$

\begin{lem}
Assume that $\bfV$ is rigid. Let $\bfv$ be a flag type such that $\calF_\bfv(\bfV)$ is not empty. Then the variety $\calF_\bfv(\bfV)$ is a smooth projective variety of dimension $\sum_{r<t}<\bfvv_r,\bfvv_t>$.
\end{lem}
\begin{proof}
The proof is similar to \cite[Cor.~4]{CaRe}.

The map 
$$
\calQ\to \prod_{i\in I}\calF_{\bfv_i}(\bbC^{v_i}),\qquad (X,\phi)\to \phi
$$ 
is clearly a vector bundle of rank

$$
\sum_{h\in H}\sum_{r=1}^{d}\bfvv_{s(h),r}\cdot\bfv_{t(h),r}=\sum_{h\in H}\sum_{r\geqslant t}\bfvv_{s(h),r}\cdot\bfvv_{t(h),t}.
$$

On the other hand, the map 
$$
\pi\colon\calQ\to \Rep_v(\Gamma), \qquad (X,\phi)\to X
$$ is proper. If some representation $X\in \Rep_v(\Gamma)$ is isomorphic to $\bfV$, then the fiber $\pi^{-1}(X)$ is isomorphic to $\calF_\bfv(\bfV)$. Since $\bfV$ is rigid, the subset of representations in $\Rep_v(\Gamma)$ isomorphic to $\bfV$ is open, see \cite[Lem.~1]{CB}. On the other hand, since we assumed that  $\calF_\bfv(\bfV)$ is not empty, we see that the image of $\pi$ is dense. Since the image of $\pi$ is closed, the map $\pi$ is surjective. Then the generic fiber of $\pi$ is smooth and has dimension $\dim\calQ-\dim{\Rep}_v(\Gamma)$.

This implies that $\calF_\bfv(\bfV)$ is a smooth projective variety of dimension
$$
\begin{array}{rcl}
&&\sum_{h\in H}\sum_{r\geqslant t}\bfvv_{s(h),r}\cdot\bfvv_{t(h),t}+\dim \prod_{i\in I}\calF_{\bfv_i}(\bbC^{v_i})-\dim \Rep_v(\Gamma)\\
&=&\sum_{h\in H}\sum_{r\geqslant t}\bfvv_{s(h),r}\cdot\bfvv_{t(h),t}+\sum_{i\in I}\sum_{r<t}\bfvv_{i,r}\cdot\bfvv_{i,t}\\
&-&\sum_{h\in H}\sum_{r,t}\bfvv_{s(h),r}\cdot \bfvv_{t(h),t}\\
&=&\sum_{r<t}<\bfvv_r,\bfvv_t>.\\
\end{array}
$$
\end{proof}

\subsection{Diagonal decomposition}

In is proved in \cite[Thm.~37]{IEFR} that quiver Grassmannians associated to rigid representations have zero odd cohomology groups. The goal of this section is to prove the same statement for flag generalizations of quiver Grassmannians. In fact, \cite[Thm.~37]{IEFR} proves a stronger property, that they call property $(S)$, see \cite[Def.~11]{IEFR}. The key point of their proof is that to check property $(S)$, it is enough to construct a diagonal decomposition.

Let $X$ be a smooth complete complex variety. We denote $A^*(X)$ the Chow ring of $X$. Let $\Delta\subset X\times X$ be the diagonal.  Denote by $\pi_1,\pi_2\colon X\times X\to X$ the two natural projections.

\begin{df}
We say that $X$ has a \emph{diagonal decomposition} if the class $[\Delta]$ of the diagonal $\Delta\subset X\times X$ in $A^*(X\times X)$ has the following decomposition
\begin{equation}
\label{eq:diag-dec}
[\Delta]=\sum_{j\in J}\pi_1^*\alpha_j\cdot \pi_2^*\beta_j,
\end{equation}
where $J$ is a finite set and $\alpha_j,\beta_j\in A^*(X)$.
\end{df}
It is proved in \cite[Thm.~2.1]{ES} that if $X$ has a diagonal decomposition, then it satisfies property $(S)$. Then, in particular, we have $H^{\rm odd}(X,\bfk)=0$.

For a vector bundle $\calE$ on $X$ we denote by $c_i(\calE)$ its $i$th Chern class, $c_i(\calE)\in A^i(X)$. Denote by $c(\calE,t)$ the Chern polynomial $c(\calE,t)=\sum_i c_i(\calE)t^i$.

\begin{thm}
\label{thm:rigid}
Assume that $\bfV\in\Rep(\Gamma)$ is rigid and that $\bfv$ is a flag type. Then $\calF_\bfv(\bfV)$ is either empty or has a diagonal decomposition. In particular, for every ring $\bfk$, we have $H^{\rm odd}(\calF_\bfv(\bfV),\bfk)=0$.
\end{thm}
\begin{proof}
The proof is similar to \cite[Thm.~37]{IEFR}. Similarly to their proof, we are going to construct a vector bundle on  $\calF_\bfv(\bfV)\times \calF_\bfv(\bfV)$ of rank $(\dim \calF_\bfv(\bfV))$ that has a section that is zero exactly on the diagonal. We manage to construct such a bundle due to Proposition \ref{prop:Ext=0-bundle}.

Indeed, since $\Ext^1_\Gamma(\bfV,\bfV)=0$, we can apply Proposition \ref{prop:Ext=0-bundle} directly. We obtain a vector bundle $\calE$ of the desired dimension, its fiber over $(\bfV',\bfV'')\in \calF_\bfv(\bfV)\times \calF_\bfv(\bfV)$ is $\Hom_\eG(\bfV',\Phi(\bfV)/\bfV'')$. As above, we see $\bfV'$ and $\bfV''$ as representations of $\eG$ that are subrepresentations of $\Phi(\bfV)$. The composition $\bfV'\to \Phi(\bfV)\to\Phi(\bfV)/\bfV''$ yields a section is this vector bundle. This section is zero exactly on the diagonal. Then by \cite[Prop.~14.1, Ex.~14.1.1]{Ful}, we can describe the class of the diagonal in term of the top Chern class of the bundle: $[\Delta]=c_{\rm top}(\calE)$.

Let $D\subset A^*(\calF_\bfv(\bfV)\times \calF_\bfv(\bfV))$ be the vector subspace formed by the elements of the form as in the right hand side of (\ref{eq:diag-dec}) for $X=\calF_\bfv(\bfV)$. It is easy to see that $D$ is a subring.

We want to prove that $[\Delta]\in D$. Let us prove that all the coefficients of the Chern polynomial $c(\calE,t)$ are in $D$. Recall from (\ref{eq:ses-bundle-final}) that $\calE$ is a part of a short exact sequence of vector bundles $0\to\calE\to\calF\to\calG\to0$, where 

$$
\calF={\bigoplus_{r\in[1;d]}}\calF_r,\qquad \mbox{$\calF_r$ has fiber $\Hom_{\Gamma}(\bfV'_r,\bfV/\bfV''_r)$}
$$
and
$$
\calG=\bigoplus_{r\in[1;d-1]}\calG_r,\qquad \mbox{$\calG_r$ has fiber $\Hom_{\Gamma}(\bfV'_r,\bfV/\bfV''_{r+1})$}.
$$

So we have $c(\calE,t)=c(\calF,t)\cdot c(\calG,t)^{-1}$. In particular, it is enough to prove that the coefficients of $c(\calF_r,t)$ and $c(\calG_r,t)$ are in $D$.

Let us show this for $\calF_r$, the proof for $\calG_r$ is similar.
By (\ref{eq:ses-Hom-quiver}), the bundle $\calF_r$ is a part of a short exact sequence
$$
0\to\calF_r\to \calX_r\to\calY_r\to 0, 
$$
where 
$$
\calX_r=\bigoplus_{i\in I}\calX_{i,r},\qquad \mbox{$\calX_{i,r}$ has fiber $\Hom(\bfV'_{i,r},\bfV_i/\bfV''_{i,r})$ }
$$
and 
$$
\calY_r=\bigoplus_{h\in H}\calY_{h,r},\qquad \mbox{$\calY_{h,r}$ has fiber $\Hom(\bfV'_{s(h),r},\bfV_i/\bfV''_{t(h),r})$}.
$$
Then we have $c(\calF_r,t)=c(\calX_r,t)\cdot c(\calY_r,t)^{-1}$. So it is enough to prove that the coefficients of $c(\calX_{i,r},t)$ and $c(\calY_{h,r},t)$ are in $D$. This is clear from \cite{Las}, \cite[Ex.~14.5.2]{Ful}.
\end{proof}

\subsection{Type $E$}
\begin{thm}
Assume that $\Gamma$ is a Dynkin quiver (types $A$, $D$, $E$). Then, for every ring $\bfk$, we have $H^{\rm odd}(\calF_\bfv(\bfV),\bfk)=0$.
\end{thm}
\begin{proof}
For types $A$ and $D$ the statement already follows from Theorem \ref{thm:main}. 

Let us prove the statement for type $E$. As explained in Remark \ref{rk:typeE}, it is enough to check the statement only for indecomposable representations. Since indecomposable representations are rigid, the statement follows from Theorem \ref{thm:rigid}.
\end{proof}

\noindent{\it Acknowledgments}.\ The author would like to thank the anonymous referee for the careful reading of the manuscript, as well as for the propositions and comments, which have improved this paper.

\end{document}